\documentclass[12pt]{amsart}
\usepackage{amsmath,amsfonts,amsthm,amstext,amscd}
\usepackage{graphicx}
\usepackage{array}
\usepackage{hyperref}
\usepackage{newlfont}
\usepackage{cancel} 

\def\cal#1{\mathcal{#1}}
\def\bb#1{\mathbb{#1}}

\def\bf#1{\mathbf{#1}}

\usepackage[all]{xy}
\newcommand{\bbR}{\mathbb{R}}
\newcommand{\calQ}{\mathcal{Q}}

 \newtheorem{thm}{Theorem}[section]
 \newtheorem{cor}[thm]{Corollary}
 \newtheorem{lem}[thm]{Lemma}
 \newtheorem{prop}[thm]{Proposition}
 \theoremstyle{definition}
 \newtheorem{defn}[thm]{Definition}
 \theoremstyle{remark}

 \numberwithin{equation}{section}
 \DeclareMathOperator{\Hom}{Hom}
 \DeclareMathOperator{\image}{image}
 \DeclareMathOperator{\rank}{rank}

\begin{document}

%
%
%
%
%
%
%
%
%

\title[Geometry of higher order variational identities]
 {The projective symplectic geometry of higher order variational problems:\\ 
 	minimality conditions}

\author[Dur\'an]{C. Dur\'an}

\address{C. Dur\'an\\Departamento de Matem\'atica, UFPR \\
 Setor de Ci\^encias Exatas, Centro Polit\'ecnico,
 Caixa Postal 019081,  CEP 81531-990,
 Curitiba, PR, Brazil}

\email{cduran@ufpr.br}

\author[Otero]{D. Otero}

\address{D. Otero\\Departamento de Matem\'atica, UFPR \\
	Setor de Ci\^encias Exatas, Centro Polit\'ecnico,
	Caixa Postal 019081,  CEP 81531-990,
	Curitiba, PR, Brazil and 
	IME-USP, Rua do Matão, 1010, CEP 05508-090, S\~ao Paulo, SP}

\email{diego.mano@gmail.com}

\thanks{The second named author was financially supported by CNPq, grant number 140837/2012-4.}
\subjclass{Primary  53C05}

\keywords{Calculus of variations, isotropic subspaces, Lagrangian Grassmannian}

\date{}

\begin{abstract}
	We associate  curves of isotropic, Lagrangian and coisotropic subspaces to higher order, one parameter  variational problems. Minimality and conjugacy properties of extremals are described 
	in terms of these curves.

  \end{abstract}

\maketitle
\section{Introduction}

We introduce the subject by briefly describing the classical case  of Riemannian manifolds: Let $(M,g)$ be Riemannian manifold, and $E=\frac{1}{2}g(v,v)$ its associated energy function. On the space $\Omega_{p,q}$ of smooth curves $\gamma$ joining 
$p=\gamma(a)$ with $q=\gamma(b)$, we consider the variational problem with fixed endpoints associated to the action functional 
$A(\gamma) = \int_a^b E(\gamma'(s))\, ds$. The local minimality of a geodesic $\gamma$ is assured, using the second variation formula,  by the strong positive definiteness of the Hessian of the functional, which in turn is given by the non-existence of 
 conjugate points on the interval $(a,b]$.

It is well-known that the constructions mentioned in the previous paragraph can be studied 
via the local, global and self-intersection properties of certain curves of Lagrangian planes in a symplectic vector space. This viewpoint, in addition to be of great value in 
	understanding purely Riemannian geometry \cite{paternain,ruggiero}, is particularly suited to the study 
of more general variational problems, e.g., Finsler geometry \cite{javaloyes-vitorio},  Lorentzian and semi-Riemannian geometry \cite{mercuri-piccione-tausk} and sub-Riemannian geometry \cite{agrachev}.

The aim of this paper is to develop this projective symplectic viewpoint for 
{\em higher order} variational problems, that is, functionals associated to 
		Lagrangians $L:J^k(\bbR,M) \to \bbR$ defined on the space of $k$-jets of curves on $M$, using a purely variational approach: directly from the Hessian of the given functional and its associated Euler-Lagrange equation, without detouring through non-holonomic mechanics. Higher order variational problems 
		appear, for example in control theory (e.g. \cite{piccione-control}) and 
		applications to physics (e.g., higher-order mechanics considered in the book \cite{leon-rodriguez} and in the papers \cite{crampin-sarlet-cantrijn, martinez-roy}).
		 
		 We use, extend and interpret in the symplectic-projetive setting the line of classical identities of 
		Cimino, Picone, Eastham, Eswaran, etc.,\cite{cimino,eastham,eswaran,kreith,leighton,picone}, which systematize the writing 
		of higher order Lagrangians as a perfect square plus a total differential term.  For the higher dimensional case, non-commutativity 
		produces combinatorial difficulties in the establishment of the aforementioned identities. In this case we extend to arbitrary dimensions the work of Coppel \cite{coppel}, which uses the 
		Legendre transform to produce a Hamiltonian version that is much more amenable to 
		computation. As a by-product, we get a version of Eswaran's identity for the higher
		dimensional case.   
		
		We shall see that to such variational problems there is a curve of {\em isotropic} subspaces whose successive prolongations determine the conjugacy, positive definiteness of the Hessian and, 
		therefore, the local minimality of the extremals.

As in the first order problem, the main actor for determining minimality and conjugacy 
is the self-intersection properties of a (prolonged) curve of {\em Lagrangian} subspaces which we call the {\em Jacobi curve}. However, in
contrast with the first-order case, this curve is {\em always} degenerate, even for 
positive definite Lagrangians. 
Another purpose of this work is to precisely state the degree of degeneracy of 
the Jacobi curve for higher-order problems. The present paper focuses on projectivization of the sufficient conditions for minimality; more precise information concerning the index of the Hessian (that in particular furnishes neccessary conditions) will be studied in the sequel \cite{duran-eidam-otero}.

The paper is organized as follows:  a brief review  of the preliminary setting up is done in section 
\ref{preliminaries}.
The core of the paper is then given in sections \ref{one-dimension} and \ref{n-dimension}, where the Jacobi curve is constructed and studied respectively for one-dimensional
and $n$-dimensional, higher order variational problems. The reason for the separation
is that the one-dimensional case can be treated using the one-dimensional classical
identities, which makes it a good starting point to understand the structure of the 
Jacobi curve, whereas in the high dimensional case we shall use the Legendre 
transform in order to apply the work of Coppel \cite{coppel}, which we extend to the $n$-dimensional
case. 

\newpage

\section{Preliminaries}\label{preliminaries}

\subsection{The Hessian of a higher order variational problem} \label{higher-order-manifolds}

Let $M^n$ be an $n$-dimensional manifold, and $L:J^k(\bbR, M) \rightarrow \bb{R}$ be a $k$-th order 
Lagrangian. We shall assume the {\em strong Legendre condition}, which is that the restriction 
of the Lagrangian $L$ to the inverse image of the projection 
$J^k(\bbR, M) \to J^{k-1}(\bbR, M)$ is strictly convex. Such a Lagrangian takes the local form $L(q_0^A,\dots,q_k^A)$,  where 
the subindex is the order of differentiation and the superindex the $n$-coordinates 
on $M$ gives local coordinates on $J^k(\bbR, M)$, and the strong Legendre condition 
takes the form ``$\frac{\partial L}{\partial q_k^A \partial q_k^B}$ is positive definite". From this data, we associate 
the functional 
\[
\mathcal L [\sigma] = \int_a^b L(j^k(\sigma)(s))ds \, ,
\]
where $j^k \sigma (s)$ denotes the $k$-jet prolongation of $\sigma$; in the coordinates 
given this is just the curve $(\sigma(s), \dot\sigma(s), \dots \sigma^{(k)}(s))$. Take $p,q \in J^{k-1}(\bbR, M)$, and let
us, as usual, restrict $\mathcal L$ to the set of curves $\Omega_{p,q}[a,b]$ given by 
\[
\Omega_{p,q}[a,b] = \{\sigma:[a,b] \rightarrow M \, | \, j^{k-1} \sigma (a) = p \text{ and } j^{k-1} \sigma (b) = q\}.
\]

Consider now an extremal $\gamma(t)$ of the functional $\mathcal L$. This curve 
satisfies the Euler-Lagrange equation, which might be understood in local coordinates
(giving an elementary approach as in \cite{gelfand-fomin}, 2.11), a symplectic approach as in \cite{leon-rodriguez,tul} or 
a more abstract one as in \cite{anderson}. The approach taken makes no difference 
in what follows. For the standard infinite dimensional manifold theory needed to talk about 
differentiability  we refer the reader to \cite{kriegl-michor, palais-book}.

The crucial point is that, if $\gamma$ is a critical point of $\mathcal L$, that is, 
$d\mathcal L_\gamma = 0$ for some reasonable notion of differentiability, then the Hessian of $\mathcal L$ at $\gamma$ is intrinsically 
defined, for example as in \cite{palais-article1,palais-article2}. This Hessian is a quadratic form on the space 
$T_\gamma \Omega_{p,q}[a,b]$
of vector fields along $\gamma(t)$ which vanish up to order $k-1$ at the endpoints. 
Since we know {\em a priori} that the Hessian is well defined, we can compute in a convenient way: let us choose a frame 
$v_1(t), \dots , v_n(t)$ of vector fields along $\gamma(t)$ that span 
$T_{\gamma(t)}M$ at each point. In terms of this frame, each element  
$X(t) \in  T_{\gamma(t)}M$ can be written as 
$X(t)= \sum_i h_i(t) v_i(t)$, and then $ T_\gamma \Omega_{p,q}[a,b]$ can be identified 
(with respect to the frame $v_1, \dots , v_n$) with the space 
\[
C^\infty_{k}([a,b], \bbR^n)=
\{h \in C^\infty([a,b],\bbR^n) \, |\, h^{(s)}(a) = h^{(s)}(b) = 0, \text{ for } 0 \leq s \leq k-1  \}
\]
then we can work in local coordinates do show that the Hessian of $\mathcal L$ at an
extremum is given by a quadratic form as follows:
\[
\calQ[h] = \int_a^b  \sum_{1\leq i \leq j \leq k} 
 h^{(j)}(t)^\top Q_{ij}(t)h^{(i)}(t) \, ,
\]
where each $Q_{ij}(t)$ 
is a smooth $n\times n$-matrix valued function. The matrices $Q_{ij}$ depend on the Lagrangian 
and its derivatives, but the only important information is that the top level $Q_{kk}(t)$ is symmetric and positive definite. These two conditions  are independent 
of the chosen frame; if the frame comes from prolonged local coordinates, the term 
$Q_{kk}(t)$ is essentially the Hessian of the Lagrangian restricted to the coordinates $q_k^A$.

The aim of this paper is, starting from the classical identities of higher order calculus of variations, to {\em construct and study the projective curves whose 
	lack of self-intersection imply that $\calQ$ is positive definite.} Let us remark
that the absence of self-intersection of our Jacobi curves on $(a,b]$ will imply 
positive definiteness, but a simple step similar to sections 28 and 29.4 of \cite{gelfand-fomin} allows the passage from positive definite to {\em strongly} positive definite (that 
 is, $\calQ(h) \geq C|h|$ for some constant $C>0$ and an appropriate $C^k$-norm 
 for $h$).
 Thus the projective topology of 
the Jacobi curves (plus the approximation up to order 2 of the functional $\mathcal L$)
does indeed provide sufficient conditions for local minimality.

\subsection{Curves in the half - and divisible - Grassmannians and their rank} \label{previous}

As mentioned in the introduction, Jacobi curves in the context of higher-order 
variational problems will always be degenerate. In this section we will furnish the 
tools to quantify this degeneracy.

The  local geometry (as opposed 
to the global topology) of curves in Grassmann manifolds
can be studied in terms of 
frames spanning the given curve. This is the approach 
used in  \cite{alvarez-duran} and \cite{duran-peixoto}, where the reader 
can find details about the construction of local invariants. 

In this section we use this viewpoint to 
{\em rank} of a curve in the Grassmannian, which measures its degeneracy.

Recall that if $V$ is a real vector space, the tangent space of a 
Grassmann manifold $Gr(k,V)$ can be canonically 
identified with the quotient vector space $T_\ell Gr(k,V) \cong \Hom(\ell,V/\ell)$. 

\begin{defn}
Let $I$ be an interval and $\ell:I \to Gr(k,V)$ be a curve in a Grassmann manifold. The {\em rank} of $\ell$ at $t\in I$ is the rank of $\ell'(t)$ considered as an element of 
$\Hom(\ell(t),V/\ell(t))$.  
\end{defn}

We now fix a basis on $V$ so that the $V \cong \bbR^n$ and denote the Grassmannian 
by $Gr(k,n)$. Given a curve $\ell(t) \in Gr(k,n)$, we can lift it to a curve 
of $k$ linearly independent vectors $a_1(t),\dots a_k(t)$ in $\bbR^n$ that span $\ell(t)$. This we can codify as a $n\times k$ matrix $\mathcal A(t)$ whose columns 
are the vectors $a_i(t)$, $\mathcal A(t) = (a_1(t)| \dots | a_k(t))$ (the vertical bars denote juxtaposition) We have

\begin{prop} \label{rank-in-frames}
	In the situation above, the rank of $\ell(t)$ is given by  $\rank(\mathcal A(t) | \mathcal A'(t)) - k$.
\end{prop}

\begin{proof}
Let us recall concretely how the identification $T_\ell Gr(k,V) \cong \Hom(\ell,V/\ell)$
is done, in a way that is useful for computations involving the frame once a basis of 
$V$ is fixed: given a curve $\ell(t)$, choose a curve of idempotent matrices
$\rho(t)$ representing projections 
(not necessarily orthogonal, since we do not assume any Euclidean structure in $V$) 
such that the image of $\rho(t)$ is $\ell(t)$. Then we have:

\begin{enumerate}
	\item The  derivative of a curve of projections $\rho(t)$ is a curve of endomorphisms that maps $\image(\rho(t))$ into $\ker(\rho(t))$ and vice-versa.
	\item The quotient $V/\ell(0)$ can be identified with $\ker \rho(0)$.
\end{enumerate}

Then, the derivative $\rho'(t_0)$ provides a map 
form $\image(\rho(t_0)) = \ell(t_0)$ into $\ker \rho(t_0) \cong V/\ell(t_0)$. 
It is straightforward, after unraveling the identifications, that this map is 
independent of the curve of projections chosen to represent $\ell(t)$. 

Fix one such adapted curve of projection matrices $\rho(t)$. Then, since 
$\rho(t) \mathcal A(t) = \mathcal A(t)$, we have 
\[
\mathcal A'(t) = \rho'(t)\mathcal A (t) + \rho(t) \mathcal A'(t) \, .
\] 
Juxtaposing and computing the rank,  we get 
\[
\rank (\mathcal A (t) |\mathcal A'(t) )= 
\rank (\mathcal A (t) | \rho'(t)\mathcal A (t) + \rho(t) \mathcal A'(t))  = 
\rank (\mathcal A (t) | \rho'(t) \mathcal A(t)) \, ,
\]
where the last equality follows from  
$\image(\rho(t)) = \ell(t)$ and therefore the columns  $\rho(t) \mathcal A'(t)$ are in
the span of the columns before the vertical bar. Now since the columns of $\mathcal A(t)$ are in 
the image of $\rho(t)$ and the columns $\rho'(t)\mathcal A (t)$ are in the kernel of 
$\rho(t)$, it follows that the rank of the right-hand side of the equation above 
is $k+\rank \rho'(t)\mathcal A (t) = k + \rank \ell(t)$.
\end{proof}

In the case of the half-Grassmannian $Gr(k,2k)$  a curve 
$\ell(t)$ is called {\em fanning} if it has maximal rank, that is the rank of 
$\ell(t)$ is $k$ and $\ell'(t): \ell(t) \to \bbR^{2k}/\ell(t)$ is invertible. This (generic) 
condition is the 
starting point of the invariants defined in \cite{alvarez-duran}: the invariantly 
defined nilpotent endomorphism $\bf{F}(t)$ of $\bbR^{2k}$ given by the composition 
$\bbR ^{2k} \to \bbR^{2k}/\ell \stackrel{(\ell'(t))^{-1}}{\to} \ell(t) \hookrightarrow \bbR^{2k}$. 

All of the invariants studied in \cite{alvarez-duran} for the Grassmannian can be 
used for the local study of curves in the Lagrangian Grassmannian; there is
a single extra {\em discrete} invariant, the {\em signature}, that solves 
the equivalence problem for curves in the Lagrangian Grassmannian.

It is important to note that, in the case of higher order variational problems, 
we will construct a curve of Lagrangian subspaces whose self-intersections 
(or, rather, lack thereof) control
the positivity of the second variation. However, in the context of higher order 
variational problems, this curve is {\em never} fanning in the sense of \cite{alvarez-duran}
thus a different approach to the the invariants must be taken. One possibility 
is using the theory developed by \cite{zelenko}.
An alternative approach could be the use of curves $\ell(t)$ in a {\em divisible Grassmannian} $Gr(k,nk)$ and 
the invariant theory that has been developed in \cite{duran-peixoto}. This is 
especially adapted to the projective geometry of higher order linear differential 
equations \cite{wil}, such as the Euler-Lagrange equation of a quadratic Lagrangian of higher order. The 
concept of fanning curve in this case requires the $(n-1)$-jet of the curve: if 
$\ell(t)$ is a curve in a divisible Grassmannian $Gr(k,nk)$  spanned as above 
by the columns of a $nk \times k$ matrix $\mathcal A(t)$, we say that $\ell$ is 
fanning if the matrix $(\mathcal A(t) | \mathcal A'(t) | \dots | \mathcal A^{(n-1)}(t))$ is invertible. This is, again, a generic condition on smooth 
curves in a divisible Grassmannian, and that is satisfied in our case of 
{\em isotropic} curves whose prolongation gives the curve in the Lagrangian 
Grassmannian.

 An important construction for these
curves is the {\em canonical linear flag} (compare \cite{zelenko})
\[
Span\{A(t)\} \subset Span\{A(t),\dot{A}(t)\}\subset  \dots \, ,
\]
which only depends on the curve $\ell(t)$ and its successive jets. We refer to each step of this sequence of inclusions as a {\em prolongation} of the curve $\ell$. 

Fanning curves 
in the divisible Grassmannian satisfy that the canonical linear flag jumps 
dimension by $k$ at each stage up to the forced stabilization at maximal 
dimension $nk$.

\section{The symplectic projective geometry of higher order quadratic functionals: the one dimensional case} \label{one-dimension}

Let $a<b \in \bbR$. Consider the general quadratic $k$-th order functional
\[
\calQ[h] = \int_a^b  \sum_{1\leq i \leq j \leq k} Q_{ij}(t)h^{(i)}(t) h^{(j)}(t) \, dt,
\]
defined on the subspace of $C^\infty([a,b])$ consisting of functions  vanishing up to order $k-1$ at $a$ and $b$.  
By repeated integration by parts, $\calQ$ can be written as\footnote{In general, this is not possible in the higher dimensional case. Compare footnote 18 in chapter 5 of  \cite{gelfand-fomin}.} 

\begin{equation} \label{eq:funcional_quadratico}
\calQ[h] = \int_a^b   P_0(t)h(t)^2 + \dots P_1(t)(\dot h(t))^2 + \dots P_k(t)(h^{(k)})^2\, dt .
\end{equation}

The Euler-Lagrange equation of $\calQ$ is given by 
\begin{equation}\label{eq:jacobi1}
P_0 h-\frac{d}{dt}(P_1\dot{h})+\ldots+
(-1)^{k}\frac{d^{k}}{dt^{k}}(P_{k}h^{(k)})=0.
\end{equation}

We call this equation the {\em Jacobi equation} of the functional $\calQ$, or, more generally, of a functional whose Hessian is given by $\calQ$. The Jacobi equation 
is an actual differential equation of order $2k$ if $P_k(t) \neq 0$ for all $t\in [a,b]$;  we shall assume the {\em strict Legendre condition:} $P_k(t)$ is actually 
positive on the interval $[a,b]$.

\subsection{Eswaran identities} \label{sec-eswaran}
Let  $\{\sigma_i\, , \, i = 1, \ldots, k\}$ be a set of linearly independent solutions of \eqref{eq:jacobi1} satisfying
\begin{equation}\label{eq:condicoes_iniciais1}
\sigma_i^{(j)}(a) = 0, \text{ for } i = 1,2,\ldots, k, \text{ and } j = 0,1,2,\ldots, k-1,
\end{equation}
and consider the  ``sub-Wronskian"
\begin{equation}\label{eq:subWronskian}
W[\sigma_1, \sigma_2, \ldots, \sigma_k](t) = \det
\begin{pmatrix}
\sigma_1 & \sigma_2 & \cdots & \sigma_k\\
\dot{\sigma_1} & \dot{\sigma_2} & \cdots & \dot{\sigma_k}\\
\vdots & \vdots & \cdots & \vdots\\
{\sigma_1}^{(k-1)} & {\sigma_2}^{(k-1)} & \cdots & {\sigma_k}^{(k-1)}
\end{pmatrix}.
\end{equation}

\begin{defn}
	A point $t^\ast \in (a,b]$ is said to be {\em conjugate} to $a$ if 
	$W[\sigma_1, \sigma_2, \ldots, \sigma_k](t^\ast) = 0$.
\end{defn}

We have have then the following theorem:

\begin{thm}(Eswaran) \label{positivity-dim1}
	If there are no conjugate points to $a$ on $(a,b]$, then the functional $\calQ$ is 
	positive definite, that is, $\calQ(h) \geq 0$ and $\calQ(h)=0$ only when $h\equiv 0$. 
\end{thm}

The main idea behind Eswaran's result is the following identity
which allow us, under the disconjugacy hypothesis, to write the quadratic functional 
$\calQ$ as a perfect square plus a total differential:
\begin{thm}(Eswaran, following Eastham) \label{eswaran-identities}
	Let $\sigma_1,\ldots, \sigma_k$ a linearly independent set of solutions of \eqref{eq:jacobi} satisfying \eqref{eq:condicoes_iniciais1} such that the  sub-Wronskian satisfies $W[\sigma_1,\sigma_2,\ldots,\sigma_k] \neq 0$ in 
	the interval $(a;b)$. Then for any $h\in C^k([a,b]),$ we have the identity
	\begin{equation}\label{eq:identidade_picone1}
	\sum_{l=0}^k P_l(t) (h^{(l)})^2 = P_k\left(\frac{W[h,\sigma_1,\sigma_2,\ldots,\sigma_k]}{W[\sigma_1,\sigma_2,\ldots,\sigma_k]}\right)^2 + \frac{dR}{dt},
	\end{equation}
	where 
	\[
	W[h,\sigma_1,\sigma_2,\ldots,\sigma_k] = \det
	\begin{pmatrix}
	h & \sigma_1 & \sigma_2 & \cdots & \sigma_k\\
	\dot{h} &\dot{\sigma_1} & \dot{\sigma_2} & \cdots & \dot{\sigma_k}\\
	\vdots & \vdots & \vdots & \cdots & \vdots\\
	h^{(k-1)} & {\sigma_1}^{(k-1)} & {\sigma_2}^{(k-1)} & \cdots & {\sigma_k}^{(k-1)}\\
	h^{(k)} & {\sigma_1}^{(k)} & {\sigma_2}^{(k)} & \cdots & {\sigma_k}^{(k)}
	\end{pmatrix},
	\]
	where $R$ is a rational expression in $\sigma_i$, $h^{(i)}$ and $P_i$, such that $R(t^\ast) = 0$ if $h^{(i)}(t^\ast) = 0$ for all $i=0,1,2,\ldots k-1$.
\end{thm}

Thus, under the hypothesis  of theorem \ref{eswaran-identities}, we can rewrite  
$\calQ$ as 
\[
\calQ (h) = \int_a^b \sum_{l=0}^k P_l(t) h^{(l)} dt = \int_a^b \left(P_k\left(\frac{W[h,\sigma_1,\sigma_2,\ldots,\sigma_k]}{W[\sigma_1,\sigma_2,\ldots,\sigma_k]}\right)^2 + \frac{dR}{dt}\right) dt,
\]
from which theorem \ref{positivity-dim1} immediately follows. 

These kind of identities were studied by Cimino and Picone \cite{cimino,picone} in the 
order one case (second order Euler-Lagrange equation),  Leighton and Kreith 
\cite{leighton,kreith}
for second order case (fourth order Euler-Lagrange equation), and in for general order 
(one-dimensional) variational problems by 
Eastham \cite{eastham}.

Easwaran \cite{eswaran} uses the identity developed in \cite{eastham} to attain the minimality conditions in the one dimensional case. Coppel \cite{coppel} also develops the minimality conditions for extemals of higher order one-dimensional variational problems assisted by the Legendre transform, in order to be in the context of linear Hamiltonian systems. In section \ref{n-dimension}, we observe that, by being careful about the order of multiplication of certain matrices,  Coppel's approach can be extended to arbitrary finite dimensional problems.

\subsection{Jacobi curves}

Let us note that theorem \ref{positivity-dim1} is projective in two senses: first,
if we choose a different set of linearly independent solutions, say 
$\eta_i(t) = \sum a_{ij}\sigma_j(t)$ for some constant invertible $k\times k$ matrix, 
the conjugacy condition is the same. This means that the conjugacy condition depends 
only on the subspace (of the space of all solutions) defined by the vanishing of 
the first $k$ derivatives, which we call the {\em vertical} subspace. 

More importantly, we do not need to have  actual solutions but only their (simultaneous) projective class:
if $\phi:[a,b] \to \bbR$ is a never-vanishing differentiable function  and 
we substitute each $\sigma_i(t)$ by $\eta_i(t)=\phi(t)\sigma_i(t)$, 
then $W[\eta_1, \eta_2, \ldots, \eta_k](t)= \phi^k(t) 
W[\sigma_1, \sigma_2, \ldots, \sigma_k](t)$ and therefore their zeros coincide. 

This motivates us to consider the following moving frame in $\bbR^{2k}$: let 
$\sigma_1(t),\dots, \sigma_k(t),\sigma_{k+1}(t), \dots , \sigma_{2k}(t)$ be
linearly independent solutions of the Euler-Lagrange equations \eqref{eq:jacobi}, where
the first $k$ solutions vanish up to order $k-1$ at $t=a$, and 
\[
\mathcal{C}(t) = \begin{pmatrix} \sigma_1(t) \\ \vdots\\  \sigma_k(t) \\ \sigma_{k+1}(t) \\ \vdots \\ \sigma_{2k}(t) \end{pmatrix} \, .
\]

Let now $p(t)$ denote the class of $\mathcal C$ under the projection to $\bbR P^{2k-1}$. We have 

\begin{prop} \label{is_fanning}
	The curve $p(t)$ is fanning as a curve in the divisible Grassmannian.
\end{prop}

\begin{proof}
	This follows directly from the fact that the determinant of the juxtaposed matrix 
$(\mathcal{C}(t) | \mathcal{C}'(t) | \dots | \mathcal{C}^{2k-1}(t))$ is the Wronskian 
of the $2k$ linearly independent solutions of the Jacobi equation. 
	
\end{proof}

We now consider the  frame obtained by the $k-1$ prolongation of $\mathcal C$, 
\[
\mathcal{A}(t) = 
\begin{pmatrix} \sigma_1(t) & \dot\sigma_1(t) & \dots & \sigma_1^{(k-1)} \\
\vdots               & \vdots       &        & \vdots \\
\sigma_k(t) & \dot\sigma_1(t) & \dots & \sigma_k^{(k-1)} \\
\sigma_{k+1}(t) & \dot\sigma_{k+1}(t) & \dots & \sigma_{k+1}^{(k-1)} \\ 
\vdots               & \vdots       &        & \vdots \\
\sigma_{2k}(t) & \dot\sigma_{2k}(t) & \dots & \sigma_{2k}^{(k-1)} \end{pmatrix},
\]

\begin{defn}
	The {\em Jacobi curve} $\ell(t)$ is the space spanned by the columns of $\mathcal A(t)$.
\end{defn}
It follows from theorem \ref{is_fanning} that $\ell(t)$ is $k$-dimensional, that is, 
$\ell$ is a curve in th half-Grassmannian $Gr(k,2k)$. Now if we define the {\em vertical space} $\mathcal V \subset \bbR^{2k}$ as the vectors having vanishing first $k$-coordinates, we have that $\mathcal V = \ell(a)$ and 
theorem \ref{positivity-dim1} translates to 

\begin{thm}
	If $\ell(t^\ast) \cap \mathcal V = \{0\}$ for all $t^\ast \in (a,b]$, then the functional $\calQ(t)$ is positive definite. 
\end{thm}

In contrast to the first order variational problems, we have that the curve 
$\ell(t)$ is {\em not} fanning. Indeed, 

\begin{prop}
	The rank of the curve $\ell(t)$ is one. 
\end{prop}

\begin{proof}
	It follows immediately by computing the rank of the extended matrix 
	$(\mathcal A(t) \, | \, \dot{\mathcal  A(t)})$; the columns $k,k+1,\dots,2k-1$ are 
	repeats of the columns $2,\dots,k$ but the first $k$ columns and the last 
	column are linearly independent by proposition \ref{is_fanning}.
\end{proof}

An important feature of the curve $p(t)$ is its {\em symplectic} behavior. This is better observed  
after applying the Legendre transform which will be done, in general dimension, in the next section. We shall see then that the canonical flag given by the successive prolongations of $p(t)$ are 
isotropic-Lagrangian-coisotropic, according to the dimension. 

\section{The symplectic projective geometry of higher order quadratic functionals: the general case} \label{n-dimension}

We now consider a quadratic functional 
\[
\calQ = \int_a^b L(t, \dot h(t), \ddot h(t), \dots h^{(k)}(t)) \, dt
\]
 where 
 \[
 L(t,h,\dot h, \dots h^{(k)})= \sum_{1\leq i \leq j \leq k} 
 h^{(i)\top}  M_{ij}(t)h^{(j)} \,
 \] 
and now $h:[a,b] \to \bbR^n$ is a vector-valued smooth function,
vanishing up to order $k-1$ at the ends of the interval, as in  section 
\ref{higher-order-manifolds}.  In general, non-commutativity will 
not let us transform $\calQ$ into a functional of the form \eqref{eq:funcional_quadratico}. We shall only assume that $M_{jj}(t)$ is symmetric (which is harmless since $ v^\top A v $ vanishes if $A$ is antisymmetric) 
and the {\em strong Legendre condition:} $M_{kk}(t)$, in addition to  being symmetric, 
{\em positive definite} for all $t\in [a,b]$. It is easy to see that by integrating 
by parts, we can make $M_{ij}(t) = 0$ if $j > i + 1$ or if $j < i$ , and then $L$ can be written 
as 
\[
L(t,h,\dot h, \dots h^{(k)}) = \frac{1}{2}\sum_{i = 0}^k {h^{(i)}}^T M_{ii} h^{(i)} + \sum_{i = 0}^{k-1}  {h^{(i)}}^T M_{i(i+1)} h^{(i+1)}, 
\]
where $M_{ii} = M_{ii}^T$ and $M_{kk}$ is positive definite.
This reduction 
greatly simplifies the computations. From now on we drop the independent variable $t$ 
from the notation. 

The Euler-Lagrange equation (which we again call  Jacobi, equation, if the 
quadratic functional is the Hessian of a functional along an extremal) is then 
\begin{equation}\label{eq:jacobi}
\frac{\partial {L}}{\partial q_0} - 
\frac{d}{dt} \left(\frac{\partial {L}}{\partial q_1}\right) + 
\frac{d^2}{dt^2} \left(\frac{\partial {L}}{\partial q_2}\right) + \ldots +
(-1)^{k}\frac{d^{k}}{dt^{k}} \left(\frac{\partial {L}}{\partial q_k} \right) = 0.
\end{equation}
Where the partial derivatives above can be easily obtained from the form of  ${L}$:
\begin{gather*}
\frac{\partial {L}}{\partial q_0} = M_{00} h + M_{01} \dot{h}, \\
\frac{\partial {L}}{\partial q_j} = M_{jj} h^{(j)} + M_{(j-1)j}^T h^{(j-1)} + M_{j(j+1)} h^{(j+1)},\, \text{ for } j = 1,\ldots, k-1, \\
\frac{\partial {L}}{\partial q_k} = M_{kk} h^{(k)} + M_{(k-1)k}^T h^{(k-1)} \, .
\end{gather*}

These relations can be written in the following form
\[
\begin{pmatrix} \frac{\partial L}{\partial q_0} \\ \vdots \\ \frac{\partial L}{\partial q_{k-1}}\end{pmatrix} = 
C(t) \begin{pmatrix} h \\ \vdots \\ h^{(k-1)} \end{pmatrix} + \begin{pmatrix} 0 \\ \vdots \\ M_{(k-1)k}M_{kk}^{-1}\frac{\partial L}{\partial q_k} \end{pmatrix},
\]
where $C$ is a curve of $kn \times kn$ matrices	that have the $M_{ij}$ blocks arranged as 
\[
C(t) = 
\begin{pmatrix}
M_{00} & M_{01} & 0 & 0 & \cdots & 0 \\
M_{01}^T & M_{11} & M_{12} & 0 & \cdots & 0 \\
0 & M_{12}^T & M_{22} & M_{23} & \cdots & 0  \\
\vdots & \vdots & \ddots & \ddots & \ddots & \vdots &\\
0 & \cdots & 0 & M_{(k-3)(k-2)}^T & M_{(k-2)(k-2)} & M_{(k-2)(k-1)}^T &\\
0 & \cdots & 0 & 0 & M_{(k-2)(k-1)}^T & \tilde{M}_{(k-1)(k-1)} 
\end{pmatrix},
\]
and $\tilde{M}_{(k-1)(k-1)} = M_{(k-1)(k-1)} - M_{(k-1)k}M_{kk}^{-1}M_{(k-1)k}^T$.
\subsection{The Legendre transform and Hamiltonian presentation}

As remarked as the end of section \ref{sec-eswaran},
a Hamiltonian version of the Jacobi equation,  using a higher dimensional extension of the methods of \cite{coppel}, will greatly simplify the computations needed to establish the 
Jacobi curve and its relationship with minimality. 
 
Consider then the {\em Legendre transform} 
$\text{Leg}:\mathbb{R}^{2kn}\rightarrow\mathbb{R}^{2kn}$ associated to $L$. The Legendre transform applied to the $(2k-1)$-jet of a curve $h$ is given by:
\[
\text{Leg}(h, \dot{h}, \ldots, h^{(k-1)}, h^{(k)}, \ldots, h^{(2k-1)}) = (y,z) 
\]
with
\[
y = \begin{pmatrix} h \\ \dot{h} \\ \ddot{h} \\ \vdots \\ h^{(k-1)} \end{pmatrix} \quad \text{and} \quad
z = \begin{pmatrix} z_1 \\ z_2 \\ z_3 \\ \vdots \\ z_k \end{pmatrix} 
\]
where
\[
z_i = \sum_{j=i}^k(-1)^{j-i} \left(\frac{d}{dt}\right)^{j-i}\left(\frac{\partial L}{\partial q_j}\right).
\]

Computation shows that the Legendre transform will take solutions of \eqref{eq:jacobi} into solutions of a Hamiltonian system in the variables $(y,z)$ given by
\begin{equation}\label{eq:hamiltonian_system}
\begin{pmatrix} \dot{y}\\ \dot{z} \end{pmatrix} = 
\begin{pmatrix} A(t) & B(t)\\ C(t) & -A(t)^T \end{pmatrix}
\begin{pmatrix} y \\ z \end{pmatrix},
\end{equation}
where $C(t)$ is the same as above, $A(t)$ is given by
\[
A(t) = 
\begin{pmatrix}
0 & \text{Id} & 0 & \cdots & 0 & 0\\
0 & 0 & \text{Id} & \cdots & 0 & 0\\
\vdots & \vdots & \vdots & \ddots & \vdots & \vdots\\
0 & 0 & 0 & \cdots & \text{Id} & 0\\
0 & 0 & 0 & \cdots & 0 & \text{Id}\\
0 & 0 & 0 & \cdots & 0 & -M_{kk}^{-1} M_{(k-1)k}^T
\end{pmatrix},
\]
and $B(t)$ is given by
\[
B(t) = 
\begin{pmatrix}
0 & 0 & 0 & \cdots & 0 & 0\\
0 & 0 & 0 & \cdots & 0 & 0\\
\vdots & \vdots & \vdots & \vdots & \vdots & \vdots\\
0 & 0 & 0 & \cdots & 0 & 0 \\
0 & 0 & 0 & \cdots & 0 & 0 \\
0 & 0 & 0 & \cdots & 0 & (M_{kk})^{-1}
\end{pmatrix},
\]
where the $\text{Id}$ and $0$ blocks have size $n \times n$ in both matrices above.

Note that $B(t)^T = B(t)$ e $C(t)^T = C(t)$, and then the matrix with the $A$, $B$ and $C$  blocks in the expression \eqref{eq:hamiltonian_system} will be in the Lie algebra of 
the symplectic group (with respect to the canonical symplectic form on $\mathbb{R}^{2n}$ making the $(y,0)$ and $(0,z)$ Lagrangian subspaces). Thus, equation 
\eqref{eq:hamiltonian_system} is indeed a Hamiltonian system.

Let us define another transform on the $k$-jets of functions $h$, which we call the 
{\em zeroing transform} since it maps the functional $\calQ$ to a functional that 
apparently involves no derivatives.

Under the zeroing transform, the $k$-jet of $h$ is mapped to  $(\hat y,\hat z)$, where
\begin{equation}\label{eq:legendre_k_jet}
\hat y = \begin{pmatrix} h \\ \dot{h} \\ \vdots \\ h^{(k-2)} \\ h^{(k-1)} \end{pmatrix} \quad \text{and} \quad
\hat z = \begin{pmatrix} 0 \\ 0 \\ \vdots \\ 0 \\ M_{kk}h^{(k)} + M_{(k-1)k}^Th^{(k-1)} \end{pmatrix},
\end{equation}

There are two important properties of the zeroing transform: first, even tough the pair $(\hat y,\hat z)$ usually does not satisfies equation
\ref{eq:hamiltonian_system} the equation
$
\dot{\hat y} = A(t) \hat y + B(t) \hat z(t)
$
is an {\em identity} implied by the form of the matrices $A$ and $B$. This identity is needed in the generalized Picone identity \ref{generalizedPiconeTheorem}.

The zeroing transform also allows the writing of the functional $\calQ$ in a simpler way. Again take a function $h \in C^k([a,b], \bb{R}^n)$ and the pair $(\hat y,\hat z)$ defined by \eqref{eq:legendre_k_jet}. Then we will have
\begin{align*}
& {\hat z}^T B {\hat z} = \begin{pmatrix} {\hat z}_1^T & \cdots & {\hat z}_k^T \end{pmatrix}
\begin{pmatrix}
0 & 0 & 0 & \cdots & 0 & 0\\
0 & 0 & 0 & \cdots & 0 & 0\\
\vdots & \vdots & \vdots & \vdots & \vdots & \vdots\\
0 & 0 & 0 & \cdots & 0 & 0 \\
0 & 0 & 0 & \cdots & 0 & 0 \\
0 & 0 & 0 & \cdots & 0 & (M_{kk})^{-1}
\end{pmatrix}
\begin{pmatrix} {\hat z}_1 \\ \vdots \\ {\hat z}_k \end{pmatrix} = \\
& = {\hat z}_k^T M_{kk}^{-1} {\hat z}_k
= \left({h^{(k)}}^T M_{kk} + {h^{(k-1)}}^T M_{(k-1)k}\right) M_{kk}^{-1} \left(M_{kk} h^{(k)} + M_{(k-1)k}^T h^{(k-1)}\right) = \\
& = {h^{(k)}}^T M_{kk} h^{(k)} + 2{h^{(k-1)}}^T M_{(k-1)k} h^{(k)} + {h^{(k-1)}}^T M_{(k-1)k} M_{kk}^{-1} M_{(k-1)k}^T h^{(k-1)},
\end{align*}
and
\begin{align*}
{\hat y}^T C {\hat y} & =
{\hat y}^T \left(
\begin{pmatrix} \frac{\partial L}{\partial q_0} \\ \vdots \\ \frac{\partial L}{\partial q_{k-1}}\end{pmatrix} - 
\begin{pmatrix} 0 \\ \vdots \\ M_{(k-1)k}M_{kk}^{-1}\frac{\partial L}{\partial q_k} \end{pmatrix}\right) = \\
& = {\hat y}^T \begin{pmatrix} \frac{\partial L}{\partial q_0} \\ \vdots \\ \frac{\partial L}{\partial q_{k-1}}\end{pmatrix} - {\hat y}^T \begin{pmatrix} 0 \\ \vdots \\ M_{(k-1)k}M_{kk}^{-1}\frac{\partial L}{\partial q_k} \end{pmatrix} = \\
& = \sum_{i = 0}^{k-1} {h^{(i)}}^T M_{ii} h^{(i)} + 2 \sum_{i = 0}^{k-2}  {h^{(i)}}^T M_{i(i+1)} + {h^{(k-1)}}^T M_{(k-1)k} h^{(k)} + \\
& - {h^{(k-1)}}^T M_{(k-1)k} M_{kk}^{-1} \frac{\partial L}{\partial q_k} = \\
& =  \sum_{i = 0}^{k-1} {h^{(i)}}^T M_{ii} h^{(i)} + 2 \sum_{i = 0}^{k-2}  {h^{(i)}}^T M_{i(i+1)} h^{(i+1)} + {h^{(k-1)}}^T M_{(k-1)k} h^{(k)} \\
& - \left( {h^{(k-1)}}^T M_{(k-1)k} h^{(k)} + {h^{(k-1)}}^T M_{(k-1)k} M_{kk}^{-1} M_{(k-1)k}^T h^{(k-1)} \right) = \\
& = \sum_{i = 0}^{k-1} {h^{(i)}}^T M_{ii} h^{(i)} + 2 \sum_{i = 0}^{k-2}  {h^{(i)}}^T M_{i(i+1)} h^{(i+1)} - {h^{(k-1)}}^T M_{(k-1)k} M_{kk}^{-1} M_{(k-1)k}^T h^{(k-1)}.
\end{align*}

Adding ${\hat z}^T B {\hat z}$ and ${\hat y}^T C {\hat y}$ leads to
\[
{\hat z}^T B {\hat z} + {\hat y}^T C {\hat y} = \sum_{i = 0}^{k} {h^{(i)}}^T M_{ii} h^{(i)} + 2 \sum_{i = 0}^{k-1}  {h^{(i)}}^T M_{i(i+1)} h^{(i+1)} = 2 L.
\]

Then, the Legendre transform  takes solutions of \eqref{eq:jacobi} to solutions of \eqref{eq:hamiltonian_system}, and the zeroing transform  will give the equality of the quadratic functionals (up to a factor 2):
\[
\tilde{\calQ} = \int_a^b {\hat z}^T B {\hat z} + {\hat y}^T C {\hat y} \, dt = 2 \calQ = 2 \int_a^b L(t, \dot h(t), \ddot h(t), \dots h^{(k)}(t)) \, dt.
\]

Let $h_1, \ldots, h_{kn}$ be a linearly independent set of solutions of \eqref{eq:jacobi} such that all the derivatives up to the $k-1$ order vanish at $t = a$. Consider the following sub-Wronskian
\[
W[h_1, \ldots, h_{kn}](t) = \det
\begin{pmatrix}
h_1^T(t) & \dot{h_1}^T(t) & \cdots & {h_1^{(k-1)}}^T(t) \\
h_2^T(t) & \dot{h_2}^T(t) & \cdots & {h_2^{(k-1)}}^T(t) \\
\vdots & \vdots & \vdots & \vdots \\
h_{kn}^T(t) & \dot{h_{kn}}^T(t) & \cdots & {h_{kn}^{(k-1)}}^T(t)
\end{pmatrix}
\]

\begin{defn}
	A point $t^\ast \in (a,b]$ is said to be {\em conjugate} to $a$ if 
	$W[h_1, \ldots, h_{kn}](t^\ast) = 0$.
\end{defn}

Then the following is immediate:
\begin{lem}
A point $t^\ast$ is conjugate to $a$ if, and only if, exists a non-trivial solution $h$ of \eqref{eq:jacobi} such that $h^{(i)}(a) = h^{(i)}(t^\ast) = 0$, for $i = 0, 1, \ldots, k-1$.
\end{lem}

In the same way as 1 dimensional case we will have

\begin{thm}\label{thm:positivity_any_dimension}
	If there are no conjugate points to $a$ on $(a,b]$, then the functional $\calQ$ is 
	positive definite, that is, $\calQ(h) \geq 0$ and $\calQ(h)=0$ only when $h\equiv 0$. 
\end{thm}

The proof of this theorem will involve a generalized Picone identity of the Hamiltonian system \eqref{eq:hamiltonian_system}. Let us fix a set $\{h_1, \ldots, h_{kn}\}$ of linearly independent solutions of \eqref{eq:jacobi} such that all the derivatives up to the $k-1$ order vanish at $t = a$, and  consider the image of the Legendre transform of the $(2k-1)$-jet of each $h_j$,
\[
\text{Leg}\left(h_j, \dot{h_j}, \ldots, h_j^{(2k-1)}\right) = \begin{pmatrix} \mu_j \\ \zeta_j \end{pmatrix},
\]
we now put then together by constructing the $2kn \times kn $ matrix where the columns are the image of the Legendre transform for each $h_j$:
\[
\begin{pmatrix} Y(t) \\ Z(t) \end{pmatrix} = 
\begin{pmatrix}
\mu_1 & \cdots & \mu_j \\
\zeta_1 & \cdots & \zeta_j 
\end{pmatrix},
\]
which defines
\[
Y(t) = \begin{pmatrix} \mu_1 & \cdots & \mu_j \end{pmatrix} \quad \text{and} \quad
Z(t) = \begin{pmatrix} \zeta_1 & \cdots & \zeta_j \end{pmatrix}.
\]
We have the following lemma, where the proof follows by the fact that each image of the Legendre transform of $h_j$ satisfy \ref{eq:hamiltonian_system} and using the initial condition at $t = a$:
\begin{lem} \label{hipoPicone}
In the conditions above we have
\[
Y(t)^T Z(t) - Z(t)^T Y(t) = 0, \, \forall t.
\]
\end{lem}

\begin{cor}
	In the conditions above, and supposing that $Y(t)$ is invertible for all $t$, we have that $Z(t)Y(t)^{-1}$ is symmetric for all $t$.
\end{cor}

We can now state the generalized Picone identity:

\begin{thm}[Generalized Picone Identity] \label{generalizedPiconeTheorem}
Let $(Y;Z)$ satisfy \ref{hipoPicone} above and suppose that $Y(t)$ is invertible for all $t \in (a;b]$. Also consider the pair $(y; z)$ satisfying the relation $\dot{y} = A(t) y + B(t) z$. In these conditions we have
\begin{equation} \label{eq:identidade_picone2}
\frac{d}{dt}(y^T ZY^{-1}y) =z^T B z + y^T C y  - (z-ZY^{-1}y)^TB(z-ZY^{-1}y).
\end{equation}
\end{thm}

\begin{proof}
	Calculating the derivative above, using $\dot{y} = A(t) y + B(t)z$ and
	\[
	\dot{Y} = AY + BZ, \quad \dot{Z} = CY - A^T Z,
	\]
	we  have
	\begin{align*}
	& \frac{d}{dt} (y^T ZY^{-1}y) = {\dot{y}}^T ZY^{-1}y + y^T \dot{Z}Y^{-1}y + y^T Z(-Y^{-1}\dot{Y}Y^{-1})y + y^T ZY^{-1}\dot{y} = \\
	& = (Ay + Bz)^T ZY^{-1}y + y^T (CY - A^TZ)Y^{-1}y - y^T Z(Y^{-1}(AY+BZ)Y^{-1})y + \\ 
	& + y^T ZY^{-1}(Ay+Bz) = \\
	& = (\cancel{y^T A} + z^TB) ZY^{-1}y + y^T (C - \cancel{A^TZY^{-1}})y - y^T Z(\cancel{Y^{-1}A} + Y^{-1}BZY^{-1})y + \\ 
	& + y^T ZY^{-1}(\cancel{Ay}+Bz) = \\
	& = z^TB ZY^{-1}y + y^T C y - y^TZ Y^{-1}BZY^{-1}y + y^T ZY^{-1} Bz + z^T B z - z^T B z = \\
	& = z^T B +  y^T C y + \underbrace{(z^TB ZY^{-1}y - y^TZ Y^{-1}BZY^{-1}y + y^T ZY^{-1} Bz - z^T B z)}_\Delta.
	\end{align*}
	
	Expanding the last term in \eqref{eq:identidade_picone2} and using that $ZY^{-1}$ is symmetric we have
	\begin{align*}
	& - (z-ZY^{-1}y)^TB(z-ZY^{-1}y) = (y^T(ZY^{-1})^T-z^T)B(z-ZY^{-1}y) = \\
	& = (y^TZY^{-1}-z^T)B(z-ZY^{-1}y) = y^TZY^{-1}Bz - y^TZY^{-1} B ZY^{-1}y - z^T B z + \\ 
	& + z^T B ZY^{-1}y = \Delta,
	\end{align*}
	and then, the identity follows.
\end{proof}

With the hypothesis above, let $h \in C^k([a,b],\bbR^n)$  satisfying $h^{(i)}(a) = h^{(i)}(b) = 0$, $0 \leq i \leq k-1$, and consider the image $(\hat y, \hat z)$ of the zeroing transform of the $k$-jet of $h$ (as in \eqref{eq:legendre_k_jet}). We have
\begin{align*}
2 \calQ[h] = \tilde{\calQ}[\hat y, \hat z] & = \int_a^b {\hat z}^T B {\hat z} + {\hat y}^T C {\hat y} \, dt = \\ 
& = \int_a^b \frac{d}{dt}({\hat y}^T ZY^{-1}{\hat y}) + ({\hat z}-ZY^{-1}{\hat y})^TB({\hat z}-ZY^{-1}{\hat y}) \, dt = \\
& = \int_a^b ({\hat z}-ZY^{-1}{\hat y})^TB({\hat z}-ZY^{-1}{\hat y}) \, dt \geq 0
\end{align*}
and
\begin{align*}
2 \calQ & = \tilde{\calQ} = 0 \Leftrightarrow  \int_a^b ({\hat z}-ZY^{-1}{\hat y})^TB({\hat z}-ZY^{-1}{\hat y}) \, dt = 0 \Leftrightarrow \\
& \Leftrightarrow B({\hat z}-ZY^{-1}{\hat y}) = 0 \Leftrightarrow \dot{{\hat y}} - A{\hat y} - BZY^{-1}{\hat y} = 0 \Leftrightarrow \\
& \Leftrightarrow \dot{{\hat y}} = (A + BZY^{-1}){\hat y} \Leftrightarrow {\hat y} \equiv 0. 
\end{align*}

Then theorem \ref{thm:positivity_any_dimension} follows, since $\hat y\equiv 0$ implies 
that $h\equiv0$. 

\subsection{Jacobi curves}

Let us projectivize as in the 1-dimensional case. 
 Define the  $2kn \times n$ frame
\begin{equation}\label{eq:frame_of_solutions}
\cal{A}(t) = 
\begin{pmatrix} h_1^T \\ \vdots \\ h_{kn}^T \\ h_{kn+1}^T \\ \vdots \\ h_{2kn}^T \end{pmatrix},
\end{equation}
where $\{h_1, \ldots, h_{2kn}\}$ is a fundamental set of solutions of \eqref{eq:jacobi} such that for $i = 1, \ldots, kn,$ each $h_i$ has all derivatives vanishing up to order $k-1$ at $t=a$.

Again considering the space $p(t)$ spanned by the columns of $\cal{A}$ at each $t$ we will have 

\begin{thm}\label{thm:is_generalized_fanning}
The curve $p(t)$ is a fanning curve in the Grassmannian $\text{Gr}(n, 2kn)$
\end{thm}

\begin{proof}
	Considering the prolongation to $(2k-1)$-jet of $p$ written in terms of the frame $\cal{A}$, the rank of this prolongation is maximal, i.e., for each $t$ the $2kn \times 2kn$ matrix bellow is non-degenerate:
	\[
		\left(\cal{A}(t) \Big| \dot{\cal{A}}(t) \Big| \cdots \Big| \cal{A}^{(2k-1)}(t)\right).
	\]
	The non-degeneracy comes from the fact that the determinant of the matrix above is the Wronskian of a set of fundamental solutions of \eqref{eq:jacobi} (which is non-zero for all $t$), and then the assertion follows.
	
\end{proof}

\begin{defn}
	The {\em Jacobi curve} $\ell:[a,b] \rightarrow \text{Gr}(kn, 2kn)$ is the $(k-1)$-jet prolongation of $p$, that is, the curve of subspaces spanned by the columns of the matrix
	\begin{equation}\label{eq:jacobi-frame}
		\left(\cal{A}(t) \Big| \dot{\cal{A}}(t) \Big| \cdots \Big| \cal{A}^{(k-1)}(t)\right).
	\end{equation}
\end{defn}

Now if we define the {\em vertical space} $\mathcal V_{kn}^{2kn} \subset \bbR^{2k}$ as the vectors having vanishing first $kn$-coordinates, we have that $\mathcal V_{kn}^{2kn} = \ell(a)$ and 
theorem \ref{thm:positivity_any_dimension} translates to 

\begin{thm}
	If $\ell(t^\ast) \cap \mathcal V_{kn}^{2kn} = \{0\}$ for all $t^\ast \in (a,b]$, then the functional $\calQ(t)$ is positive definite. 
\end{thm}

Therefore the lack of self-intersections of the Jacobi curve implies the positivity 
of the quadratic functional $\calQ$. 

In contrast with first order variational problems, this Jacobi curve is highly degenerate. We have:

\begin{thm} \label{posto1pb}
	The rank of the Jacobi curve $\ell(t)$ is $n$.
\end{thm}

\begin{proof}
Proposition \ref{rank-in-frames} can applied to the frame of \eqref{eq:jacobi-frame} of $\ell(t)$. The rank of $\ell$ is then given by 
		\[
		\mathrm{rank}
		\left(\cal{A}(t) \Big| \cdots \Big| \cal{A}^{(k-1)}(t)  \Big| \dot{\cal{A}}(t) \Big| \cdots \Big| \cal{A}^{(k)}(t)\right) - kn\, ,
		\]
which, by theorem \ref{thm:is_generalized_fanning}, is $n(k+1) - kn = n$. 
\end{proof}

Observe that the rank of a fanning curve in the half-Grassmannian $ \text{Gr}(kn, 2kn)$ is $kn$; thus the Jacobi curve of higher order variational problems ($k>1$) is {\em never} fanning.

\subsection{Isotropic-Lagrangian-Coisotropic Flag}

We now study the symplectic properties of the fanning curve $p:[a,b] \rightarrow \text{Gr}(n, 2kn)$  and its prolongations. Endow $\mathbb{R}^{2kn}$ with the canonical 
symplectic form $\omega_{\text{can}}$ induced by the decomposition
$\mathbb{R}^{2kn} \cong  \mathbb{R}^{kn} \oplus \mathbb{R}^{kn}$ in its first and last 
$kn$ coordinates. Then we have

\begin{thm}\label{thm:symp-frame}
	Considering the symplectic space $(\mathbb{R}^{2kn}, \omega_{\text{can}})$, the curve $p:[a,b] \rightarrow \text{Gr}(n, 2kn)$ defined in the last section satisfies
	\begin{itemize}
		\item $j^{i} p:[a,b] \rightarrow \text{Gr}((i+1)n, 2kn)$ is a curve of isotropic subspaces for $i = 0, \ldots, k-2$,
		\item $ \ell = j^{k-1} p:[a,b] \rightarrow \text{Gr}(kn, 2kn)$ is a curve of Lagrangian subspaces,
		\item $j^{i} p:[a,b] \rightarrow \text{Gr}((i+1)n, 2kn)$ is a curve of coisotropic subspaces for $i = k, \ldots, 2k-1$.
	\end{itemize}
\end{thm}

\begin{proof}
	The proof relies on the invariance of the isotropic-Lagrangian-coisotropic concepts 
	under certain ``tiltings" of the Legendre transform. Recall from the beginning of this section that the Legendre transform has the form 
\[
\text{Leg}(h, \ldots, h^{(2k-1)}) = \begin{pmatrix} y \\ z \end{pmatrix} =
\begin{pmatrix} \text{Id} & 0 \\ B_1 & B_2 \end{pmatrix} \begin{pmatrix} h \\ \vdots \\ h^{(2k-1)} \end{pmatrix}.
\]
Where we now have to study in more detail the matrices $B_1$ and $B_2$. 
In the expression above the blocks $\text{Id}$, $0$, $B_1$ and $B_2$ are $kn \times kn$ matrices. The block $B_1$ is an upper triangular matrix and the block $B_2$ has blocks $M_{kk}$ or $-M_{kk}$ in the anti-diagonal and zeros below the anti-diagonal, that is, $B_2$ is of the form
\[
B_2 = \begin{pmatrix} 
\ast & \ast & \cdots & \ast & (-1)^{k-1}M_{kk} \\
\ast & \ast & \cdots &  (-1)^{k-2} M_{kk} & 0 \\
\vdots & \vdots & \reflectbox{$\ddots$} & \vdots & \vdots \\
\ast & -M_{kk} & 0 & \cdots & 0 \\
 M_{kk} & 0 & 0 & \cdots & 0 \end{pmatrix}.
\]
From this we get that each coordinate $z_i$ will depend (linearly) only of $h^{(i-1)}, h^{(i)}, \ldots, h^{(2k-i)}$, for $i = 1, \ldots, k$.

Now for each $h_i$ given in \eqref{eq:frame_of_solutions} consider the image $(\mu_i, \zeta_i)$ of the Legendre transform of the $(2k-1)$-jet prolongation of $h_i$
\[
\text{Leg}(h_i, \ldots, h_i^{(2k-1)}) = \begin{pmatrix} \mu_i \\ \zeta_i \end{pmatrix},
\]
and construct the following matrix
\[
\begin{pmatrix} \mu_1 & \cdots & \mu_{2kn} \\ \zeta_1 & \cdots & \zeta_{2kn} \end{pmatrix} = 
\begin{pmatrix} \text{Id} & 0 \\ B_1 & B_2 \end{pmatrix}
\begin{pmatrix}
 h_1 & \cdots & h_{2kn} \\ 
\vdots & \vdots & \vdots \\
 h_1^{(2k-1)} & \cdots & h_{2kn}^{(2k-1)}
\end{pmatrix}.
\]

Calculating the transpose in the relation above will lead us to
\begin{align}\label{eq:legendre_frame}
\begin{pmatrix} \mu_1^T & \zeta_1^T \\ \vdots & \vdots \\ \mu_{2kn}^T & \zeta_{2kn}^T \end{pmatrix} & = 
\left(\cal{A}(t) \Big| \dot{\cal{A}}(t) \Big| \cdots \Big| \cal{A}^{(2k-1)}(t)\right) 
\begin{pmatrix} \text{Id} & B_1^T \\ 0 & B_2^T \end{pmatrix}\\
& = \left(\cal{A}(t) \Big| \dot{\cal{A}}(t) \Big| \cdots \Big| \cal{A}^{(k-1)}(t) \Big| \cal{C}^1(t) \Big| \cdots \Big| \cal{C}^k(t) \right) \nonumber
\end{align}
where the $2kn \times n$ blocks $\cal{C}^i$ are linear combinations of the columns of $\cal{A}^{(i-1)}, \ldots, \cal{A}^{(2k-i)}$, for $i = 1, \ldots, k$, that can be written as
\[
\cal{C}^i = \cal{A}^{(i-1)} Q_{i-1}^i + \ldots + \cal{A}^{(2k-i)} Q_{2k-i}^i,
\]
where the $Q_j^i$ are $n \times n$ matrices and, most importantly, $Q_{2k-i}^i = \pm M_{kk}$.

Now, the matrix on the left side of the equality \eqref{eq:legendre_frame} will be in the Lie group of symplectic matrices if we suppose the initial condition
\[
\begin{pmatrix} \mu_1^T(a) & \zeta_1^T(a) \\ \vdots & \vdots \\ \mu_{2kn}^T(a) & \zeta_{2kn}^T(a) \end{pmatrix} = 
\begin{pmatrix} 0_{kn} & \text{Id}_{kn} \\ -\text{Id}_{kn} & 0_{kn} \end{pmatrix} \, .
\]
Then we will have that the matrices in \eqref{eq:legendre_frame} are symplectic for each $t$ and the initial condition makes the first columns vanish up to order $k-1$ at $t=a$, as required . 

Denoting by $J$ the matrix of the canonical symplectic form $\omega_{\text{can}}$ in $\mathbb{R}^{2kn}$, we will have that
\begin{equation}\label{eq:symplectic_frame}
 \left(\cal{A}(t) \Big| \cdots \Big| \cal{C}^k(t) \right)^T J  \left(\cal{A}(t) \Big| \cdots \Big| \cal{C}^k(t) \right) = J = \begin{pmatrix} 0_{kn} & \text{Id}_{kn} \\ -\text{Id}_{kn} & 0_{kn} \end{pmatrix}.
\end{equation}
which in turn implies
\begin{equation} \label{eq:symp-frame1}
{\cal{A}^{(i-1)}}^T J \cal{A}^{(j-1)} = 0,
\end{equation}
for $i,j = 1, \ldots, k$, and
\[
\quad {\cal{C}^{i}}^T J \cal{A}^{(j-1)} = 0\,
\]
for $2 \leq i \leq k$ and $1\leq j \leq i - 1$.
Developing further the second expression using that $\cal{C}^i$ can be written as
\[
\cal{C}^i = \cal{A}^{(i-1)} Q_{i-1}^i + \ldots + \cal{A}^{(2k-i)} Q_{2k-i}^i,
\]
with $Q_{2k-i}^i = \pm M_{kk}$ (that is non-degenerate), we will have that
\begin{equation} \label{eq:symp-frame2}
{\cal{A}^{(i)}}^T J \cal{A}^{(j)} = 0,
\end{equation}
for $ k \leq i \leq 2k-2$ and $0 \leq j \leq 2k - i -2$.

Now \eqref{eq:symp-frame1} and \eqref{eq:symp-frame2} implies all items of theorem \ref{thm:symp-frame} simultaneously.

\end{proof}

\end{document}